\documentclass[11pt,twoside,a4paper]{article}
\usepackage{amsmath,amsthm,amsfonts,amscd,amssymb,bm}
\usepackage{hyperref}
\usepackage{geometry}
\usepackage{mathrsfs}
\usepackage{dsfont}
\usepackage{indentfirst}

\newtheorem{theorem}{Theorem}[section]
\newtheorem{lemma}{Lemma}[section]
\newtheorem{corollary}{Corollary}[section]
\newtheorem{remark}{Remark}[section]
\newtheorem{definition}{Definition}[section]

\begin{document}

\title{\textbf{Three series theorem for independent random variables under sub-linear expectations with applications}}
\date{}

\author{\textbf{Jiapan Xu}\ \ \ \ \ \ \ \textbf{Lixin Zhang}\\
{\small \emph{School of Mathematical Sciences, Zhejiang University, Hangzhou 310027, P. R. China}}\\
\vspace{3mm}
\emph{E-mail: statxjp@zju.edu.cn}\ \ \ \ \emph{stazlx@zju.edu.cn}}

\maketitle
\noindent{\bf{Abstract}}\ \ In this paper, motived by the notion of independent and identically distributed random variables under the sub-linear expectation initiated by Peng, we give a theorem about the convergence of a random series and establish a three series theorem of independent random variables under the sub-linear
expectations. As an application, we obtain the Marcinkiewicz's strong law of large numbers for independent and identically distributed random variables under the sub-linear expectations. The technical details are different from those for classical theorems because the sub-linear expectation and its related capacity are not additive.


\noindent{\bf{MR(2010) Subject Classification}}\ \ 60F15

\section{Introduction}

Classical limit theorems like the Kolmogorov's strong law of large numbers (SLLN for short) play an important role in the development of probability theory and its applications. The key in the proofs of these limit theorems is the additivity of probabilities and expectations. However, such additivity assumption may not be reasonable in many areas of applications since many uncertain phenomena can not be well modeled using additive probabilities or expectations. Therefore, there are some literatures about non-additive probabilities and non-additive expectations
which are useful tools for studying uncertainties in statistics, measures of risk, superhedging in finance and non-linear stochastic calculus, see \cite{Denis2006A,Gilboa1987Expected,Marinacci1999Limit,Peng S1997,peng1999monotonic,PengG-Expectation06,peng2008multi-dimensional},etc.

This paper considers the general sub-linear expectations and related non-additive probabilities generated by them. The notion of i.i.d. random variables under the sub-linear expectations was introduced by Peng  \cite{PengG-Expectation06,peng2008a,peng2009survey} and the weak convergences such as central limit theorems and weak laws of large numbers has been studied. Recently, Chen \cite{chen2010strong} obtained Kolmogorov's SLLN for independent and identically distributed (i.i.d. for short) random variables under the condition of finite $(1+\epsilon)$-moments by establishing an inequality of an exponential moment of partial sums of truncated independent random variables. The moment condition is much stronger than the one in the classical Kolmogorov's SLLN. Zhang \cite{Zhang Rosenthal} showed that Kolmogorov's SLLN holds for i.i.d. random variables under a continuous sub-linear expectation if and only if the corresponding Choquet integral is finite by using the method of picking up subsequence. Because the classical proofs of SLLN depend basically on the Kolmogorov's three series theorem, such theorems about the convergence of random variables have not been established under the sub-linear expectations. The main purpose of this paper is to establish three series theorem for independent random variables in the general sub-linear expectation spaces and use it to study the strong limit theorems. The Rosenthal's inequalities for the maximum partial sums of independent random variables under the sub-linear expectations established by Zhang \cite{Zhang Rosenthal} are our main tools to prove our results. Because the  Rosenthal's inequalities are established for random variables with zero sub-linear expectations or sub-linear  expectations near to zero and it is impossible to centralize a random variable such that both its upper expectation and lower expectation are zeros, the proof of the three series theorem is much technical.  After having the Kolmogorov's three series theorem, the idea of the proof of the Marcinkiewicz's strong law of large numbers is standard, but the  technical details are different from those for  classical Marcinkiewicz's strong law of large numbers because the sub-linear expectation and its related capacity are not additive.

The rest of the paper is organized as follows. In Section 2, we introduce the
basic notations of sub-linear expectations and capacity. We also explain the concept of identical distribution and independence in a sub-linear expectation space. In Section 3, we establish the three series theorem based on  some basic but important inequalities and lemmas. We apply it to prove the Marcinkiewicz's strong law of large numbers under sub-linear expectation. All the proofs are given in this section. Section 4 concludes with a discussion.

\section{Basic Settings}

We use the framework and notation of Peng \cite{peng2008a}. Let $(\Omega, \mathcal{F})$ be a given measurable space and let $\mathscr{H}$ be a linear space of real functions defined on $(\Omega, \mathcal{F})$ such that if
$X_1, X_2, \ldots, X_n \in \mathscr{H}$ then $\varphi(X_1, X_2, \ldots, X_n)\in \mathscr{H}$ for
each $\varphi \in C_{l,Lip}(\mathds{R_n})$, where $C_{l,Lip}(\mathds{R_n})$ denotes the linear space of (local Lipschitz) functions $\varphi$ satisfying
\begin{align*}
&|\varphi(\textbf{x})-\varphi(\textbf{y})|\leq C\left(|1+|\textbf{x}|^m+|\textbf{y}|^m\right)|\textbf{x}-\textbf{y}|,\ \ \forall\  \textbf{x},\textbf{y}\in \mathds{R_n},\\
&\text{for some}\ C>0,\ \ m\in \mathds{N} \ \text{depending on} \ \varphi.
\end{align*}
$\mathscr{H}$ is considered as a space of ``random variables". In this case, we denote $X\in \mathscr{H}$.

\begin{remark}
It is easily seen that if $\varphi_1, \varphi_2 \in C_{l,Lip}(\mathds{R_n})$, then $\varphi_1\vee \varphi_2,\  \varphi_1\wedge \varphi_2 \in C_{l,Lip}(\mathds{R_n})$ because
$\varphi_1 \vee \varphi_2 = \frac{1}{2}(\varphi_1 + \varphi_2 + |\varphi_1 - \varphi_2|)$ and $\varphi_1 \wedge \varphi_2 = \frac{1}{2}(\varphi_1 + \varphi_2 - |\varphi_1 - \varphi_2|)$. This will be very useful in our proof.
\end{remark}

\subsection{Sub-linear Expectation and Capacity}

\begin{definition}
A sub-linear expectation $\widehat{E}$ on $\mathscr{H}$ is a function $\widehat{E} : \mathscr{H}\rightarrow \overline{\mathds{R}} := [-\infty, +\infty]$ satisfying the following properties: For all $X,Y \in \mathscr{H}$, we have\\
(a) Monotonicity: If $X\geq Y$, then $\widehat{E}[X] \geq \widehat{E}[Y]$.\\
(b) Constant preserving: $\widehat{E}[c]=c$.\\
(c) Sub-additivity: $\widehat{E}[X + Y] \leq \widehat{E}[X] + \widehat{E}[Y]$ whenever $\widehat{E}[X] + \widehat{E}[Y]$ is not of the form $+\infty, -\infty$ or $-\infty, +\infty$.\\
(d) Positive homogeneity: $\widehat{E}[\lambda X]=\lambda \widehat{E}[X],\  \lambda \geq 0$.
\end{definition}

The triple $(\Omega, \mathscr{H}, \widehat{E})$ is called a sub-linear expectation space. Given a sub-linear expectation $\widehat{E}$, let us denote the conjugate expectation $\widehat{\mathcal{E}}$ of $\widehat{E}$ by
\begin{equation*}
\widehat{\mathcal{E}}[X] :=-\widehat{E}[-X],\ \ \forall \ X \in \mathscr{H}.
\end{equation*}

From the definition, it is easily shown that $\widehat{\mathcal{E}}[X]\leq \widehat{E}[X]$,\  $\widehat{E}[X + c] = \widehat{E}[X] + c$ and $\widehat{E}[X -Y]\geq \widehat{E}[X]-\widehat{E}[Y] $ for all $X, Y \in \mathscr{H}$ with $\widehat{E}[Y]$ being finite. We also call $\widehat{E}[X]$ and $\widehat{\mathcal{E}}[X]$ the upper-expectation and lower-expectation of $X$, respectively.

Next, we introduce the capacities corresponding to the sub-linear expectations.
\begin{definition}
Let $\mathcal{G}\subset \mathcal{F}.$ A function $V: \mathcal{G}\rightarrow [0, 1]$ is called a capacity if
\[
V(\emptyset)=0,\ V(\Omega)=1,\ \text{and}\  V(A)\leq V(B)\ \ \forall A\subset B,\ A,B\in \mathcal{G}.
\]
It is called to be sub-additive if $V(A\cup B)\leq V(A)+V(B)$ for all $A,\ B\in \mathcal{G}$ with $A\cup B \in \mathcal{G}.$
\end{definition}

Here we only consider the capacities generated by a sub-linear expectation. Let $(\Omega, \mathscr{H}, \widehat{E})$ be a sublinear space, and $\widehat{\mathcal{E}}$ be the conjugate expectation of $\widehat{E}$. It is natural to define the capacity of a set $A$ to be the sub-linear expectation of the indicator function $I_A$ of $A$. However, $I_A$ may be not in $\mathscr{H}$. So, we denote a pair $(\mathbb{V}, \mathcal{V})$ of capacities by
\[
\mathbb{V}(A):=\text{inf}\{\widehat{E}[\xi]:I_A\leq \xi,\ \xi \in \mathscr{H}\},\ \ \mathcal{V}(A):=1-\mathbb{V}(A^c),\ \ \forall A\in \mathcal{F},
\]
where $A^c$ is the complement set of A. Then
\begin{align*}
&\mathbb{V}(A):=\widehat{E}[I_A],\ \ \mathcal{V}(A):=\widehat{\mathcal{E}}[I_A],\ \ \text{if} \ I_A\in \mathscr{H},\\
&\widehat{E}[f]\leq \mathbb{V}(A)\leq \widehat{E}[g],\ \ \widehat{\mathcal{E}}[f]\leq \mathcal{V}(A)\leq \widehat{\mathcal{E}}[g],\ \ \text{if} \ f \leq I_A \leq g,\ \ f,g\in \mathscr{H}.
\end{align*}

It is obvious that $\mathbb{V}$ is sub-additive. But $\mathcal{V}$ and $\widehat{\mathcal{E}}$ may be not sub-additive. However, we have
\[
\mathcal{V}(A\cup B)\leq \mathcal{V}(A)+ \mathbb{V}(B)\ \
\text{and}\ \ \widehat{\mathcal{E}}[X+Y]\leq \widehat{\mathcal{E}}[X]+\widehat{E}[Y]
\]
due to the fact that $\mathbb{V}(A^c \cap B^c) =\mathbb{V}(A^c \backslash B)\geq \mathbb{V}(A^c)-\mathbb{V}(B)$ and
$\widehat{E}[-X-Y]\geq \widehat{E}[-X]-\widehat{E}[Y]$.\\

The corresponding Choquet integrals/expecations $(C_\mathbb{V}, C_\mathcal{V})$ are defined by
\[
C_V[X]=\int_{0}^{\infty}V(X\geq t)dt + \int_{-\infty}^{0}[V(X\geq t)-1]dt
\]
with $V$ being replaced by $\mathbb{V}$ and $\mathcal{V}$, respectively.

\begin{definition}
(I) A function $V: \mathcal{G}\rightarrow [0, 1]$ is called to be countably sub-additive if
\[
V(\bigcup_{n=1}^{\infty}A_n)\leq \sum_{n=1}^{\infty}V(A_n)\ \ \forall A_n\in \mathcal{F}.
\]
(II) A function $V: \mathcal{G}\rightarrow [0, 1]$ is called a continuous capacity if it satisfies:\\
(i) Continuity from below: $V(A_n)\uparrow V(A)$ if $A_n\uparrow A$, where $A_n, A\in \mathcal{F}$.\\
(ii) Continuity from above: $V(A_n)\downarrow V(A)$ if $A_n \downarrow A$, where $A_n, A\in \mathcal{F}$.
\end{definition}

\subsection{Independence and Distribution}

\begin{definition}

(\romannumeral1)(Identical distribution) Let $\textbf{X}_1$ and $\textbf{X}_2$ be two $n$-dimensional random vectors defined respectively in sub-linear expectation spaces $(\Omega_1, \mathscr{H}_1, \widehat{E}_1)$ and
$(\Omega_2, \mathscr{H}_2, \widehat{E}_2)$. They are called identically distributed, denoted by $\textbf{X}_1 \overset{d}{=}\textbf{X}_2$ if
\[
\widehat{E}_1[\varphi(\textbf{X}_1)]=\widehat{E}_2[\varphi(\textbf{X}_2)],
\ \forall\ \varphi \in C_{l,Lip}(\mathds{R_n}),
\]
whenever the sub-expectations are finite. A sequence $\{X_n; n\geq 1\}$ of random variables is said to be identically distributed if
$X_i \overset{d}{=}X_1$ for each $i \geq 1$.

\noindent(\romannumeral2)(Independence) In a sub-linear expectation space $(\Omega, \mathscr{H}, \widehat{E})$, a random vector
$\textbf{Y} = (Y_1,\ldots, Y_n)$, $Y_i \in \mathscr{H} $ is said to be independent to another random vector
$\textbf{X} = (X_1,\ldots, X_n),\ X_i \in \mathscr{H} $ under $\widehat{E}$ if for each test function $\varphi \in C_{l,Lip}(\mathds{R_m\times R_n})$ we have
\[
\widehat{E}[\varphi(\textbf{X},\textbf{Y})]
=\widehat{E}[\widehat{E}[\varphi(\textbf{x}, \textbf{Y})]|_{\textbf{x}=\textbf{X}}]
\]
whenever $\varphi(\textbf{x}):=\widehat{E}[\varphi(\textbf{x}, \textbf{Y})]< \infty$ for all $\textbf{x}$ and $\widehat{E}[|\varphi(\textbf{x})|]< \infty.$ Moreover, an independent series is a sequence of random variables $\{X_n; n\geq 1\}$ satisfying:
 $X_{i+1}$ is independent of $(X_1,\ldots,X_i)$ for each $i \geq 1$.

\end{definition}

\section{Main Result}

In this section, we give the main results. We first recall some related important inequalities and lemmas in sub-linear expectation space. Then we give a theorem about the convergence of a random series, from which we can deduce a theorem similar to Kolmogorov's three series theorem in classical probability thoery. At last, we prove the Marcinkiewicz's strong law of large numbers as an application.

\begin{lemma}\label{inequalities}
 Let $X, Y$ be real measurable random variables in sub-linear expectation space $(\Omega, \mathscr{H}, \widehat{E})$.\\
(1)\ \textbf{Holder inequality.} For $p, q>1$ with $\frac{1}{p}+\frac{1}{q}=1$, we have
\begin{equation*}
\widehat{E}|XY|\leq (\widehat{E}|X|^p)^{\frac{1}{p}}(\widehat{E}|X|^q)^{\frac{1}{q}}.
\end{equation*}

\noindent(2)\ \textbf{Chebyshev inequality.} Let $f(x)>0$ be a nondecreasing function on $\mathds{R}$. Then for any $x$,
\begin{equation*}
\mathbb{V}\left(X\geq x\right)\leq \frac{\widehat{E}\left(f(X)\right)}{f(x)}.
\end{equation*}

Let $f(x)>0$ be an even function on $\mathds{R}$ and nondecreasing on $(0,\infty)$. Then for any $x>0$,
\begin{equation*}
\mathbb{V}\left(|X|\geq x\right)\leq \frac{\widehat{E}\left(f(X)\right)}{f(x)}.
\end{equation*}

\noindent(3)\ \textbf{Jensen inequality.} Let $f(\cdot)$ be a convex function on $\mathds{R}$. Suppose that $\widehat{E}[X]$ and $\widehat{E}[f(X)]$ are finite, then
\begin{equation*}
\widehat{E}[f(X)]\geq f( \widehat{E}[X]).
\end{equation*}
\end{lemma}

The proof of Lemma \ref{inequalities} can be found in Proposition 2.1 in Chen \cite{Chen Z 2013}. Next, we give the ``the convergence part'' of the Borel-Cantelli lemma for a countably sub-additive capacity.

\begin{lemma}[Borel-Cantelli's Lemma]\label{Borel-Cantelli}
Let $\{A_n; n\geq 1\}$ be a sequence of events in $\mathcal{F}$. Suppose that $V$ is a countably sub-additive capacity. If\ $\sum_{n=1}^{\infty} V(A_n) < \infty$, then
\begin{equation*}
V(A_n\ i.o.)=0,\ \ \text{where}\ \{A_n\ i.o.\}=\bigcap_{n=1}^{\infty}\bigcup_{i=n}^{\infty}A_i.
\end{equation*}
\end{lemma}

\begin{proof}
By the monotonicity and countable sub-additivity of $V$, it follows that
\begin{equation*}
0\leq V(\bigcap_{n=1}^{\infty}\bigcup_{i=n}^{\infty}A_i)\leq
V(\bigcup_{i=n}^{\infty}A_i)\leq \sum_{i=n}^{\infty}V(A_i)\rightarrow 0\ \ \text{as}\ n \rightarrow \infty.
\end{equation*}
\end{proof}

The following theorem is the main result of this paper. It gives a criterion for the almost surely convergence of an infinite independent random series.

\begin{theorem}\label{THM1} Suppose that $\{X_n;n\geq1\}$ is a sequence of independent random variables in the sub-linear expectation space $(\Omega, \mathscr{H}, \widehat{E})$. If $\mathbb{V}$ is countably sub-additive, $\sum_{n=1}^{\infty}\hat{E}(X_n)$ and $\sum_{n=1}^{\infty}\widehat{\mathcal{E}}(X_n)$ both converge, and there exists $1\leq p \leq 2$, such that $\sum_{n=1}^{\infty}\hat{E}(|X_n|^p)<\infty$. Then $\sum_{n=1}^{\infty}X_n$ converges almost surely in capacity.
\end{theorem}

Before we pove the theorem, we need some moment inequalities for ${\mathop{\max }_{k \le {n}}}\, S_k$ which can be applied to truncated random variables freely. Here, we refer to the Rosnethal's inequality in Zhang \cite{Zhang Rosenthal}.

\begin{lemma}[Rosnethal's inequality]\label{Rosnethal}
Let $\{X_n;n\geq1\}$ be a sequence of independent random variables in the sub-linear expectation space $(\Omega, \mathscr{H}, \widehat{E})$, and denote $S_k = X_1 + \ldots + X_k, \ S_0 = 0$. If $\widehat{E}[X_k]\le 0,\ k=1,\ 2,\ \ldots\ ,$  then
\begin{equation*}
\widehat{E}\left(\left| \underset{k \le {n}}{\mathop{\max }}\, {(S_n-S_k)} \right|^p \right)
\leq 2^{2-p} \sum_{k=1}^{n}\widehat{E}\left (|X_k|^p \right),\ \ \ \  for \  1\leq p \leq 2.
\end{equation*}
\end{lemma}

If both the upper expectation $\widehat{E}[X_k]$ and the lower expectation $\widehat{\mathcal{E}}[X_k]$ are zeros, then
\begin{equation*}
\widehat{E}\left(\max_{k\le n}|S_k|^p \right)\le
2 \sum_{k=1}^{n}\widehat{E}\left (|X_k|^p \right),\ \ \ \  for \  1\leq p \leq 2.
\end{equation*}
The proof of the convergence of the series $\sum_{n=1}^{\infty}X_n$ is then standard and very similar to that for random variables in a classical probablility space. Howerver, it is impossible to centralize a random variable such that both its upper expectation  and the lower expectation are zero. We shall consider $(S_n-S_k)^+$ and $(S_n-S_k)^-$ respectively.

\begin{lemma}\label{positve}
 Suppose that $x \text{and}\  y \in \mathds{R}$, then
\begin{equation*}
(x+y)^+\leq x^{+}+|y|\ \ \text{and}\ \ (x+y)^-\leq x^{-}+|y|.
\end{equation*}
\end{lemma}

\begin{proof}
As for $(x+y)^+$, considering the following two cases:\\
(\romannumeral1) If $x+y\leq 0$, then $(x+y)^+ = 0 \leq x^{+}+|y|$.\\
(\romannumeral2) Suppose $x+y> 0$, then $(x+y)^+ = x+y \leq x^{+} + y^{+}\leq x^{+} + |y|$.

As for $(x+y)^-$, using the first inequality $(x+y)^+\leq x^{+}+|y|$ we have proved, we can get $(x+y)^- =(-x-y)^+\leq (-x)^{+}+|-y|= x^{-}+|y|$.
\end{proof}

\begin{proof}[Proof of Theorem \ref{THM1}]
Define $S_n = \sum_{i=1}^{n}X_i$.

\textbf{Step 1:} Prove $\{S_n;{n \geq 1}\}$ is a Cauchy sequence in capacity and there exists a subsequence $\{S_{n_k};{k\geq 1}\}$ converges to some $S,\ a.s.\ \mathbb{V}$.

By the Chebyshev inequality in Lemma \ref{inequalities} with an even function $f(x)=|x|^p$, $1\leq p\leq2$ and the Cr-inequality $|x+y|^p\leq 2^{p-1}(|x|^p+|y|^p)$ , it is easy to show that for any $\epsilon > 0$, we have
\begin{align*}
\mathbb{V}\left\{|S_m-S_n|\geq\epsilon\right\}
&\leq \frac{1}{\epsilon^p}\widehat{E}\left(|S_m-S_n|^p\right)
= \frac{1}{\epsilon^p}\widehat{E}\left(\left|(S_m-S_n)^{+}-(S_m-S_n)^{-}\right|^p\right) \\
&\leq\frac{2^{p-1}}{\epsilon^p}\widehat{E}\left(\left|(S_m-S_n)^{+}\right|^p+
\left|(S_m-S_n)^{-}\right|^p\right)\\
&\leq\frac{2^{p-1}}{\epsilon^p}\left\{\widehat{E}\left(\left|(S_m-S_n)^{+}\right|^p\right)+
\widehat{E}\left(\left|(S_m-S_n)^{-}\right|^p\right)\right\}.
\end{align*}

 Define $T_n = \sum_{i=1}^{n}\widehat{E}(X_i)$, $Y_j=X_j-\widehat{E}(X_j)$ and $\widetilde{S}_n = \sum_{j=1}^{n}Y_j$. It is straightforward to show that $\widehat{E}(Y_j)=0$ and $\widetilde{S}_{n}=S_{n}-T_{n}$. According to Lemma \ref{positve}, we have
\begin{align*}
\widehat{E}\left(\left|(S_m-S_n)^{+}\right|^p\right)
&=\widehat{E}\left(\left|\left(\widetilde{S}_m+T_m-(\widetilde{S}_n+T_n)\right)^{+}\right|^p\right)\\
&=\widehat{E}\left(\left|\left(\widetilde{S}_m-\widetilde{S}_n +T_m-T_n\right)^{+}\right|^p\right)\\
&\leq \widehat{E}\left(\left|(\widetilde{S}_m-\widetilde{S}_n)^{+}+|T_m - T_n|\right|^p\right)\\
&\leq2^{p-1}\widehat{E}\left(\left|(\widetilde{S}_m-\widetilde{S}_n)^{+}\right|^p\right)+
2^{p-1}\left|T_m - T_n\right|^p.
\end{align*}

 Because of the fact that $\{Y_n; n\geq 1\}$ is independent, from Lemma \ref{Rosnethal}, we have
\begin{align*}
\widehat{E}\left(\left|(\widetilde{S}_m-\widetilde{S}_n)^{+}\right|^p\right)
&\leq \widehat{E}\left(\underset{n\leq k \leq {m}}{\mathop{\max
}}\,\left|(\widetilde{S}_m-\widetilde{S}_k)^{+}\right|^p\right)
= \widehat{E}\left(\left|\underset{n\leq k \leq {m}}{\mathop{\max
}}\,(\widetilde{S}_m-\widetilde{S}_k)^{+}\right|^p\right)\\
&\leq \widehat{E}\left(\left|\underset{n\leq k \leq {m}}{\mathop{\max
}}\,(\widetilde{S}_m-\widetilde{S}_k)\right|^p\right)
\leq 2^{2-p}\sum_{j=n+1}^{m}\widehat{E}(|Y_j|^p).
\end{align*}

Note that
\begin{align*}
\widehat{E}(|Y_j|^p)&=
\widehat{E}\left (|X_j-\widehat{E}(X_j)|^p\right)\\
&\leq2^{p-1}\widehat{E}\left(|X_j|^p+|\widehat{E}(X_j)|^p\right)\\
&= 2^{p-1}\widehat{E}\left(|X_j|^p\right)+2^{p-1}|\widehat{E}(X_j)|^p\\
&\leq 2^{p}\widehat{E}(|X_j|^p).
\end{align*}
Here the last inequality is due to $|\widehat{E}(X_j)|^p\leq \widehat{E}\left(|X_j|^p\right) $, which can be derived from Lemma \ref{inequalities} (Jensen inequality) with the convex function $f(\cdot)=|x|^p, 1 \leq p \leq 2$. Then
\begin{align*}
\widehat{E}\left(\left|(\widetilde{S}_m-\widetilde{S}_n)^{+}\right|^p\right)
\leq 4\sum_{j=n+1}^{m}\widehat{E}(|X_j|^p)\rightarrow 0\ \ \ as\ \  m\geq n \rightarrow\infty.
\end{align*}

Meanwhile, note that $T_n = \sum_{j=1}^{n}\widehat{E}(X_j)$ is the partial sum of $\widehat{E}(X_j)$ and $\{T_n;{n\geq 1}\}$ is a convergent sequence based on the conditions in theorem \ref{THM1}. It must be a Cauchy sequence in $\mathds{R}$. Then $|T_m-T_n|\rightarrow 0$ as $m\geq n \rightarrow\infty.$ Therefore,
\begin{align*}
\widehat{E}\left(\left|(S_m-S_n)^{+}\right|^p\right)
\leq2^{p-1}\widehat{E}\left(\left|(\widetilde{S}_m-\widetilde{S}_n)^{+}\right|^p\right)+
2^{p-1}\left|T_m - T_n\right|^p \rightarrow 0\ \ \ as\ \  m\geq n \rightarrow\infty.
\end{align*}

Similarly, considering the sequence $\{-X_n;{n\geq 1}\}$, it is easy to show that
\begin{align*}
\widehat{E}\left(\left|(S_m-S_n)^{-}\right|^p\right) \rightarrow 0\ \ \ as\ \  m\geq n \rightarrow\infty.
\end{align*}
So, for any $\epsilon > 0$, as $m\geq n \rightarrow\infty $,
\begin{equation*}
\mathbb{V}\left\{|S_m-S_n|\geq\varepsilon\right\}\leq\frac{2^{p-1}}{\epsilon^p}\left\{\widehat{E}\left(\left|(S_m-S_n)^{+}\right|^p\right)+
\widehat{E}\left(\left|(S_m-S_n)^{-}\right|^p\right)\right\}\rightarrow 0.
\end{equation*}

Thus, $\{S_n; n \geq 1\}$ is a Cauchy sequence in capacity. Now we consider the subsequence $\{S_{n_k};k\geq 1\}$ satisfying
\begin{equation*}
\mathbb{V}\{|S_{n_{k+1}}-S_{n_k}|\geq \frac{1}{2^k}\}\leq \frac{1}{2^k}.
\end{equation*}

Take the set $A_k:=\{|S_{n_{k+1}}-S_{n_k}|\geq \frac{1}{2^k}\}$. Since $\mathbb{V}$ is countably sub-additive, we have
\begin{equation*}
\mathbb{V}\left( \limsup A_n\right)= \mathbb{V}\left(\bigcap_{n=1}^{\infty}\bigcup_{k=n}^{\infty}A_{k}\right)
\leq \mathbb{V}\left( \bigcup_{k=m}^{\infty}A_k\right)\leq \sum_{k=m}^{\infty}\mathbb{V}(A_k)\leq \sum_{k=m}^{\infty}\frac{1}{2^k}
\leq \frac{1}{2^{m-1}}
\end{equation*}

\noindent for each $m\geq 1.$ Letting $m\rightarrow \infty$ yields
\begin{equation*}
\mathbb{V}\left( \limsup A_n\right)= 0.
\end{equation*}

Let $\Omega_0=(\limsup A_n)^c=\bigcup_{n=1}^{\infty}\bigcap_{k=n}^{\infty}A_{k}^{c}$. It follows from the above equation that $\mathbb{V}(\Omega_0^c)=0$. For each $\omega\in \Omega_0$, there exists an integer $N(\omega)$, such that $|S_{n_{k+1}}(\omega)-S_{n_k}(\omega)|< \frac{1}{2^k}$ for any $k\geq N(\omega)$. We have $\sum_{k=1}^{\infty}|S_{n_{k+1}}(\omega)-S_{n_k}(\omega)|\leq 1 <\infty$. Then, $\{S_{n_k}(\omega);k\geq 1\}$ converges.  And $\{S_{n_k};k\geq 1\}$ converges to some $S,\ a.s.\ \mathbb{V}$, i.e.,
\begin{equation}
S_{n_k} \rightarrow S \ \ a.s.\ \mathbb{V}.
\end{equation}

\textbf{Step 2:} Prove $\underset{{{n}_{k}}<j\le {{n}_{k+1}}}{\mathop{\max }}\, \left| S_{n_{k+1}}-S_{j} \right| \rightarrow 0\ \ a.s.\ \mathbb{V}.$

According to Lemma \ref{Rosnethal} and recall the Chebyshev ineuqality, we have
\begin{align*}
 &\sum\limits_{k=1}^{\infty }{\mathbb{V}\left( \underset{{{n}_{k}}<j\le {{n}_{k+1}}}{\mathop{\max }}\, \left( S_{n_{k+1}}-T_{n_{k+1}}-(S_{j}-T_{j}) \right)^{+} \ge \epsilon  \right)} \\
 = &\sum\limits_{k=1}^{\infty }{\mathbb{V}\left( \underset{{{n}_{k}}<j\le {{n}_{k+1}}}{\mathop{\max }}\, \left( \widetilde{S}_{n_{k+1}}-\widetilde{S}_{j} \right)^{+} \ge \epsilon  \right)} \\
 = &\sum\limits_{k=1}^{\infty }{\mathbb{V}\left(\left| \underset{{{n}_{k}}<j\le {{n}_{k+1}}}{\mathop{\max }}\, \left( \widetilde{S}_{n_{k+1}}-\widetilde{S}_{j} \right)^{+} \right|\ge \epsilon \right)} \\
 \le & \sum\limits_{k=1}^{\infty }{\frac{\hat{E}\left(\left| \underset{{{n}_{k}}<j\le {{n}_{k+1}}}{\mathop{\max }}\, { (\widetilde{S}_{n_{k+1}}-\widetilde{S}_{j})^{+} }\right|^p \right)}{{{\epsilon }^{p}}}} \\
 \le & \sum\limits_{k=1}^{\infty }{\frac{\hat{E}\left(\left| \underset{{{n}_{k}}<j\le {{n}_{k+1}}}{\mathop{\max }}\, { (\widetilde{S}_{n_{k+1}}-\widetilde{S}_{j}) }\right|^p \right)}{{{\epsilon }^{p}}}} \\
 \leq & \frac{1}{\epsilon^p} \sum_{k=1}^{\infty} 2^{2-p} \sum_{j=n_k+1}^{n_{k+1}}\widehat{E}\left (|Y_j|^p \right) \\
 = &\frac{2^{2-p}}{\epsilon^p} \sum_{j=1}^{\infty} \widehat{E}\left (|Y_j|^p \right)
 \leq \frac{2^{2-p}}{\epsilon^p}\sum_{j=1}^{\infty} 2^p \widehat{E}(|X_j|^p) < \infty.
\end{align*}

Using the Borel-Cantelli's lemma, we obtain that
\begin{align}
\underset{{{n}_{k}}<j\le {{n}_{k+1}}}{\mathop{\max }}\, \left( S_{n_{k+1}}-T_{n_{k+1}}-(S_{j}-T_{j}) \right)^{+} \rightarrow 0\ \ a.s.\ \mathbb{V} \ \ \  as\ k \rightarrow \infty.
\end{align}

Note that $\{T_n;n\geq 1\}$ is a Cauchy sequence in $\textbf{R}$. That is, for any $\epsilon > 0$, there exists an integer $N$, such that for any $m, n>N$, $|T_m-T_n|< \epsilon$. So,
\begin{align}
& \underset{{{n}_{k}}<j\le {{n}_{k+1}}}{\mathop{\max }}\, \left| {T_{n_{k+1}}-T_j} \right|\rightarrow 0\ \ \ \  as\ k \rightarrow \infty.
\end{align}

From Lemma \ref{positve}, we know that
\begin{align*}
\underset{{{n}_{k}}<j\le {{n}_{k+1}}}{\mathop{\max }}\,\left( S_{n_{k+1}}-S_{j} \right)^{+}
&=\underset{{{n}_{k}}<j\le {{n}_{k+1}}}{\mathop{\max }}\,\left( S_{n_{k+1}}-S_{j}-(T_{n_{k+1}}-T_{j})+(T_{n_{k+1}}-T_{j}) \right)^{+}\\
&\leq \underset{{{n}_{k}}<j\le {{n}_{k+1}}}{\mathop{\max }}\,\left\{\left( S_{n_{k+1}}-S_{j}-(T_{n_{k+1}}-T_{j})\right)^{+}+|T_{n_{k+1}}-T_{j}|\right\}\\
&\leq\underset{{{n}_{k}}<j\le {{n}_{k+1}}}{\mathop{\max }}\,\left( S_{n_{k+1}}-T_{n_{k+1}}-(S_{j}-T_{j}) \right)^{+}+\underset{{{n}_{k}}<j\le {{n}_{k+1}}}{\mathop{\max }}\,|T_{n_{k+1}}-T_{j}|
\end{align*}

Combining (2) and (3), we can get
\begin{align*}
\underset{{{n}_{k}}<j\le {{n}_{k+1}}}{\mathop{\max }}\, \left( S_{n_{k+1}}-S_{j} \right)^{+} \rightarrow 0\ \ a.s.\ \mathbb{V}\ \ \ \ as\  k \rightarrow \infty.
\end{align*}

Similarly, considering the sequence $\{-X_n;{n \geq 1}\}$. Let $Y_j=-X_j+\widehat{\mathcal{E}}(X_j)$, $T_n = \sum_{j=1}^{n}\widehat{\mathcal{E}}(X_j)$ and $\widetilde{S}_n = \sum_{j=1}^{n}Y_j$. Then $\widehat{E}(Y_j)=0$ and $\widetilde{S}_{n}=T_{n}-S_{n}$. Using the same method as before, we have
\begin{align*}
\underset{{{n}_{k}}<j\le {{n}_{k+1}}}{\mathop{\max }}\, \left( T_{n_{k+1}}-S_{n_{k+1}}-(T_{j}-S_{j}) \right)^{+} \rightarrow 0\ \ a.s.\ \mathbb{V},
\end{align*}
which is equivalent to
\begin{align*}
\underset{{{n}_{k}}<j\le {{n}_{k+1}}}{\mathop{\max }}\, \left( S_{n_{k+1}}-T_{n_{k+1}}-(S_{j}-T_{j}) \right)^{-} \rightarrow 0\ \ a.s.\ \mathbb{V}.
\end{align*}

Using $(x+y)^-\leq x^{-}+|y|$ in Lemma \ref{positve}, we have
\begin{align*}
\underset{{{n}_{k}}<j\le {{n}_{k+1}}}{\mathop{\max }}\, \left( S_{n_{k+1}}-S_{j} \right)^{-} \rightarrow 0\ \ a.s.\ \mathbb{V}\ \ \ \ as\  k \rightarrow \infty.
\end{align*}

Thus,
\begin{align}
\underset{{{n}_{k}}<j\le {{n}_{k+1}}}{\mathop{\max }}\, \left| S_{n_{k+1}}-S_{j} \right| \rightarrow 0\ \ a.s.\ \mathbb{V}.
\end{align}

Hence, combining (1) and (4), we have
\begin{equation*}
S_{n} \rightarrow S \ \ a.s.\ \mathbb{V}.
\end{equation*}
\end{proof}

Next, for any random variable X and constant c, define $X^c=(-c)\vee(X\wedge c)$.  The following corollary follows from theorem \ref{THM1} immediately.

\begin{corollary}\label{three series} Let $\{X_n;n\geq1\}$ be a sequence of independent random variables in $(\Omega, \mathscr{H}, \widehat{E})$ and $\mathbb{V}$ is countably sub-additive, then $\sum_{n=1}^{\infty}X_n$ will converge almost surely in capacity if the following three conditions hold for some $c\in (0,+\infty)$.\\
$\begin{aligned}
&(\romannumeral1)\ \sum_{n=1}^{\infty}\mathbb{V}(|X_n|>c)<\infty. \\
&(\romannumeral2)\ \sum_{n=1}^{\infty}\hat{E}(X_n^c) \ \text{and} \ \sum_{n=1}^{\infty}\widehat{\mathcal{E}}(X_n^c)\  \text{both converge}. \\
&(\romannumeral3)\ \sum_{n=1}^{\infty}\hat{E}(|X_n^c|^q)<\infty\ \text{for some} \ 1\leq q \leq 2.
\end{aligned}$
\end{corollary}

\begin{proof}
By condition (\romannumeral1), $\sum_{n=1}^{\infty}\mathbb{V}(X_n\neq X_n^c)=\sum_{n=1}^{\infty}\mathbb{V}(|X_n|>c)<\infty$. By the Borel-Cantelli's lemma, we have $\mathbb{V}(X_n\neq X_n^c, i.o.)=0$. So, except a null set $\Omega_0$, $\sum_{n=1}^{\infty}X_n$ and $\sum_{n=1}^{\infty}X_n^c$ both converge or both diverge. According to Theorem \ref{THM1} and conditions (\romannumeral2), (\romannumeral3), $\sum_{n=1}^{\infty}X_n^c$ converges almost surely in capacity. Hence $\sum_{n=1}^{\infty}X_n$ converges almost surely in capacity.
\end{proof}

Corollary \ref{three series} is very similar to Kolmogorov's three series theorem in classical probability thoery. It gives a criterion for the almost sure convergence of an infinite series of random variables in terms of the convergence of three different series. Combined with Kronecker's lemma, it can be used to give a relatively easy proof of the Marcinkiewicz's strong law of large numbers.

\begin{theorem}\label{Marcinkiewicz}
Suppose $\{X_n;n\geq1\}$ is a sequence of independent and identical random variable in the sub-linear expectation space $(\Omega, \mathscr{H}, \widehat{E})$ and $\mathbb{V}$ is countably sub-additive. For $0<p<1$, if\ $C_\mathbb{V}(|X_1|^p)<\infty$, then $S_n/n^{1/p}\rightarrow 0\ a.s.\ \mathbb{V}.$
And for $1<p<2$, suppose $\hat{E}(X_n)=\hat{\mathcal{E}}(X_n)=\mu$ and $\lim_{a\rightarrow \infty}\widehat{E}\left((|X_1|-a)^+\right)=0$, if\ $C_\mathbb{V}(|X_1|^p)<\infty$, then $(S_n-n\mu)/n^{1/p}\rightarrow 0$  a.s. $\mathbb{V}$.
\end{theorem}

\begin{remark} If $\mathbb{V}$ is continuous and for some $\mu$, $(S_n-n\mu)/n^{1/p}\rightarrow 0$  a.s. $\mathbb{V}$, then $(X_n-\mu)/n^{1/p}\rightarrow 0$  a.s. $\mathbb{V}$. With the same arguments in the proof of Theorem 3.3 (b) of Zhang \cite{Zhang Rosenthal}, we can derive  that $C_{\mathbb{V}}(|X_1|^p)<\infty$.

\end{remark}

Before we prove the Theorem \ref{Marcinkiewicz}, we give two lemmas first. One is the well-known Kronecker Lemma.
\begin{lemma}[Kronecker] \label{Kronecker}
Suppose $\{x_n\}_{n=1}^{\infty}$ and $\{a_n\}_{n=1}^{\infty}$ are two infinite sequences of real numbers, $0<a_n\uparrow \infty$, if $\sum_{n=1}^{\infty}x_n/a_n$ converges, then $\sum_{i=1}^{n}x_i/a_n \rightarrow 0.$
\end{lemma}

\begin{lemma}\label{square converge} Let X be a random variable defined on a sublinear expectation space $(\Omega, \mathscr{H}, \widehat{E})$. Suppose $C_\mathbb{V}(|X|^p)<\infty,\ 0<p<2$, then
\begin{align*}
\sum_{n=1}^{\infty}\frac{\hat{E}\left(\left(|{X}|\wedge n^{1/p}\right)^2\right)}{n^{2/p}}<\infty.
\end{align*}
\end{lemma}

\begin{proof}
Since $\left(|{X}|\wedge n^{1/p}\right)^2\leq n^{2/p}$, we have $\lim_{a\rightarrow \infty}\widehat{E}\left(\left(|{X}|\wedge n^{1/p}\right)^2\wedge a\right)=\widehat{E}\left(\left(|{X}|\wedge n^{1/p}\right)^2\right)$.\\
According to the Lemma 3.9(b) in Zhang \cite{Zhang Rosenthal}, we have
\begin{align*}
\widehat{E}\left(\left(|{X}|\wedge n^{1/p}\right)^2\right)&\leq C_\mathbb{V}\left(\left(|{X}|\wedge n^{1/p}\right)^2\right)=\int_{0}^{+\infty}\mathbb{V}\left(\left(|{X}|\wedge n^{1/p}\right)^2\geq x\right)dx\\
&=\int_{0}^{n^{2/p}}\mathbb{V}\left(\left(|{X}|\wedge n^{1/p}\right)^2\geq x\right)dx
\leq \int_{0}^{n^{2/p}}\mathbb{V}\left(X^2\geq x\right)dx \\
&=\int_{0}^{n^{2/p}}\mathbb{V}\left(|X|\geq x^{1/2}\right)dx
=\int_{0}^{n^{1/p}}2y\mathbb{V}\left(|X|\geq y\right)dy.
\end{align*}

Therefore,
\begin{align*}
\sum_{n=1}^{\infty}\frac{\hat{E}\left(\left(|{X}|\wedge n^{1/p}\right)^2\right)}{n^{2/p}}
&\leq \sum_{n=1}^{\infty} n^{-2/p}\int_{0}^{n^{1/p}}2y\mathbb{V}\left(|X|\geq y\right)dy\\
&= \sum_{n=1}^{\infty} \int_{n}^{n+1}n^{-2/p}dt\int_{0}^{n^{1/p}}2y\mathbb{V}\left(|X|\geq y\right)dy\\
&\leq c_p\int_{1}^{+\infty}t^{-2/p}\int_{0}^{t^{1/p}}2y\mathbb{V}\left(|X|\geq y\right)dydt\\
&= c_p\int_{0}^{+\infty}2y\mathbb{V}\left(|X|\geq y\right)\int_{y^p}^{+\infty}t^{-2/p}dt dy\\
&= c_p\int_{0}^{+\infty}2y\mathbb{V}\left(|X|\geq y\right)\frac{p}{2-p}y^{p-2} dy\\
&= \frac{2p\cdot c_p}{2-p}\int_{0}^{+\infty}y^{p-1}\mathbb{V}\left(|X|\geq y\right) dy\\
&= \frac{2\cdot c_p}{2-p}\int_{0}^{+\infty}\mathbb{V}\left(|X|^p\geq x\right) dx
\leq \frac{2\cdot c_p}{2-p}C_\mathbb{V}(|X|^p)<\infty,
\end{align*}
where $c_p$ is a constant only depends on $p.$
\end{proof}

\begin{proof}[Proof of Theorem \ref{Marcinkiewicz}]
 Without loss of generality, suppose $\mu=0$.

According to the Kronecker's lemma (Lemma \ref{Kronecker}), in order to prove $S_n/n^{1/p}\rightarrow 0 \ a.s.\ \mathbb{V}$, i.e., $\sum_{i=1}^nX_i/n^{1/p}\rightarrow 0 \ a.s.\ \mathbb{V},$ we only need to prove that $\sum_{n=1}^{\infty}X_n/n^{1/p}$ converges almost surely in capacity. In the following contents, we set the constant $c=1$.

\textbf{Step 1:} Prove $\sum_{n=1}^{\infty}\mathbb{V}\left(\left|{X_n}/{n^{1/p}}\right|>c\right)<\infty.$
\begin{align*}
\sum_{n=1}^{\infty}\mathbb{V}\left(\left|{X_n}/{n^{1/p}}\right|>c\right)&=
\sum_{n=1}^{\infty}\mathbb{V}(|X_n|>n^{1/p})=\sum_{n=1}^{\infty}\int_{n-1}^{n}\mathbb{V}(|X_1|>n^{1/p})dt\\
&\leq \int_{0}^{\infty}\mathbb{V}(|X_1|>t^{1/p})dt=C_\mathbb{V}(|X_1|^p)<\infty.
\end{align*}

\textbf{Step 2:} Prove $\sum_{n=1}^{\infty}\hat{E}\left((X_n/n^{1/p})^c \right)$ and $ \sum_{n=1}^{\infty}\hat{\mathcal{E}}\left((X_n/n^{1/p})^c \right)$ both converge.\\

Define $Y_n =(-n^{1/p})\vee(X_n\wedge n^{1/p})$, then $\left({X_n}/{n^{1/p}}\right)^c= {Y_n}/{n^{1/p}}$.

\vspace{2mm}
(i) For $0<p<1$, we have $\lim_{a\rightarrow \infty}\widehat{E}\left(|Y_n|\wedge a\right)=\widehat{E}\left(|Y_n|\right)$ thanks to the fact that $Y_n$ is less than or equal to $n^{1/p}$. Recall the Lemma 3.9(b) in Zhang \cite{Zhang Rosenthal}, we have
\begin{align*}
\widehat{E}\left(|Y_n|\right)&\leq C_\mathbb{V}\left(|Y_n|\right)=\int_{0}^{+\infty}\mathbb{V}\left(|Y_n|\geq x\right)dx
=\int_{0}^{n^{1/p}}\mathbb{V}\left(|Y_n|\geq x\right)dx \\
&\leq \int_{0}^{n^{1/p}}\mathbb{V}\left(|X_n|\geq x\right)dx
=\int_{0}^{n}\frac{1}{p}y^{1/p-1}\mathbb{V}\left(|X_n|^p\geq y\right)dy.
\end{align*}

It follows that
\begin{align*}
\sum_{n=1}^{\infty}\left|\hat{E}\left(\frac{X_n}{n^{1/p}}\right)^c\right|
&= \sum_{n=1}^{\infty}\frac{\left|\hat{E}({Y_n})\right|}{n^{1/p}}
\leq \sum_{n=1}^{\infty}\frac{\hat{E}(\left|{Y_n}\right|)}{n^{1/p}}\\
&\leq \frac{1}{p}\sum_{n=1}^{\infty}{n^{-1/p}}\int_{0}^{n}y^{1/p-1}\mathbb{V}\left(|X_n|^p\geq y\right)dy\\
&\leq \frac{1}{p}\sum_{n=1}^{\infty}\int_{n}^{n+1}{n^{-1/p}}dt\int_{0}^{n}y^{1/p-1}\mathbb{V}\left(|X_n|^p\geq y\right)dy\\
&\leq \frac{c_p}{p}\int_{1}^{+\infty}t^{-1/p}\int_{0}^{t}y^{1/p-1}\mathbb{V}\left(|X_n|^p\geq y\right)dy dt\\
&= \frac{c_p}{p}\int_{0}^{+\infty}y^{1/p-1}\mathbb{V}\left(|X_n|^p\geq y\right)\int_{y}^{+\infty}t^{-1/p}dt dy\\
&= \frac{c_p}{p}\int_{0}^{+\infty}y^{1/p-1}\mathbb{V}\left(|X_n|^p\geq y\right)\frac{p}{1-p}{y^{1-1/p}} dy\\
&= \frac{c_p}{1-p}\int_{0}^{+\infty}\mathbb{V}\left(|X_1|^p\geq y\right) dy=\frac{c_p}{1-p}C_\mathbb{V}(|X_1|^p)<\infty.
\end{align*}

Therefore, $\sum_{n=1}^{\infty}\hat{E}\left(\frac{X_n}{n^{1/p}}\right)^c$ converges.
On the other hand, considering the sequence $\{-X_n;n\geq 1\}$, we can easily get $\sum_{n=1}^{\infty}\widehat{\mathcal{E}}\left(\frac{X_n}{n^{1/p}}\right)^c$ converges.

(ii) For $1<p<2$, since $\lim_{a\rightarrow \infty}\widehat{E}\left((|X_n|-a)^+\right)=0$, recall the Lemma 3.9(b) in Zhang \cite{Zhang Rosenthal}, we have that for any $a>0$,
\begin{align*}
\widehat{E}\left((|X_n|-a)^+\right)\leq C_\mathbb{V}\left((|X_n|-a)^+\right)=\int_{0}^{+\infty}\mathbb{V}\left(|X_n|-a\geq x\right)dx
=\int_{a}^{+\infty}\mathbb{V}\left(|X_n|\geq x\right)dx.
\end{align*}

Using the fact that $\hat{E}(X_n)=0$, we find
\begin{align*}
\left| \widehat{E}(Y_n) \right|
&=\left| \widehat{E}(Y_n)-\widehat{E}(X_n) \right|
\leq \widehat{E}\left(|Y_n-X_n |\right)
\leq \widehat{E}\left((X_n-n^{1/p})^+\right)\\
&\leq \int_{n^{1/p}}^{+\infty}\mathbb{V}\left(|X_n|\geq x\right)dx
= \int_{n}^{+\infty}\frac{1}{p}y^{1/p-1}\mathbb{V}\left(|X_n|^p\geq y\right)dy.
\end{align*}

Then
\begin{align*}
\sum_{n=1}^{\infty}\left|\hat{E}\left(\frac{X_n}{n^{1/p}}\right)^c\right|
&= \sum_{n=1}^{\infty}\frac{\left|\hat{E}({Y_n})\right|}{n^{1/p}}
\leq \frac{1}{p}\sum_{n=1}^{\infty}{n^{-1/p}}\int_{n}^{+\infty}y^{1/p-1}\mathbb{V}\left(|X_n|^p\geq y\right)dy\\
&\leq \frac{1}{p}\sum_{n=1}^{\infty}\int_{n-1}^{n}t^{-1/p}dt\int_{n}^{+\infty}y^{1/p-1}\mathbb{V}\left(|X_n|^p\geq y\right)dy\\
&\leq \frac{1}{p}\int_{0}^{+\infty}t^{-1/p}\int_{t}^{+\infty}y^{1/p-1}\mathbb{V}\left(|X_n|^p\geq y\right)dy dt\\
&= \frac{1}{p}\int_{0}^{+\infty}y^{1/p-1}\mathbb{V}\left(|X_n|^p\geq y\right)\int_{0}^{y}t^{-1/p}dt dy\\
&= \frac{1}{p}\int_{0}^{+\infty}y^{1/p-1}\mathbb{V}\left(|X_n|^p\geq y\right)\frac{p}{p-1}{y^{1-1/p}} dy\\
&= \frac{1}{p-1}\int_{0}^{+\infty}\mathbb{V}\left(|X_1|^p\geq y\right) dy=\frac{1}{p-1}C_\mathbb{V}(|X_1|^p)<\infty.
\end{align*}

It follows that $\sum_{n=1}^{\infty}\hat{E}\left(\frac{X_n}{n^{1/p}}\right)^c$ converges.
On the other hand, since $\hat{E}(-X_n)=-\widehat{\mathcal{E}}(X_n)=0$,\ $\widehat{\mathcal{E}}(-X_n)=-\hat{E}(X_n)=0$, considering the sequence $\{-X_n;n\geq 1\}$. By the same way, we can easily get $\sum_{n=1}^{\infty}\widehat{\mathcal{E}}\left(\frac{X_n}{n^{1/p}}\right)^c$ converges.

\textbf{Step 3:} Prove $\sum_{n=1}^{\infty}\hat{E}(|(X_n/n^{1/p})^c|^2)<\infty.$\\

It follows from Lemma \ref{square converge} that
\begin{align*}
\sum_{n=1}^{\infty}\hat{E}\left(\left|\left(\frac{X_n}{n^{1/p}}\right)^c\right|^2\right)=\sum_{n=1}^{\infty}\frac{\hat{E}\left(\left(|{X_1}|\wedge n^{1/p}\right)^2\right)}{n^{2/p}}<\infty.
\end{align*}

Above all, according to Corollary \ref{three series}, $S_n/n^{1/p}\rightarrow 0 \ a.s.\ \mathbb{V}.$
\end{proof}

\section{Discussion}

In this paper, we have established a three series theorem, which gives a criterion for the almost sure convergence of an infinite independent series of random variables. We apply it to give a simple proof of the Marcinkiewicz's strong law of large numbers when $0<p<1$ and $1<p<2$ under sub-linear expectation. However, we have not give the proof of the strong law of large numbers ($p=1$). The main difficulty lies in that we can not testify the convergence of $\sum_{n=1}^{\infty}\widehat{\mathcal{E}}\left[(X_n/n)^c\right]$. I think one possible reason is that the condition may be too strong since Kolmogorov's three series theorem gives a necessary and sufficient condition of the convergence of a random series in the frame of classical probability theory, but ours only gives a sufficient condition. Another possibility is that we don't have too many tools, especially inequalities to prove the convergence of $\sum_{n=1}^{\infty}\widehat{\mathcal{E}}\left[(X_n/n)^c\right]$.

\section*{Acknowledgments}
 This work was supported by grants from the NSF of China (No. 1173101), the 973 Program (No. 2015CB352302),  Zhejiang Provincial Natural Science Foundation (No. LY17A010016) and the Fundamental Research Funds for the Central Universities.


\end{document}